\documentclass[11pt, reqno]{amsart}
\usepackage{amsmath,amssymb,amsfonts,amsthm,graphicx,mathrsfs,mathtools}
\usepackage{array}
\usepackage[margin=1.5in]{geometry}
\usepackage[dvipsnames]{xcolor}
\usepackage{graphicx}
\usepackage[safe]{tipa}
\usepackage{blkarray}
\usepackage{tikz}
\usepackage{array}
\usepackage{tikz}
\usetikzlibrary{matrix,arrows.meta}
\usetikzlibrary{matrix,arrows.meta,calc}
\usepackage{subcaption}
\makeatletter
\@namedef{subjclassname@2020}{%
\textup{2020} Mathematics Subject Classification}
\makeatother
\title[Complete embedded minimal surfaces in $\mathbb{R}^5$]{Embedded minimal surfaces with three ends associated with the Hoffman--Osserman catenoids}
\author[J. Lee]{Jaehoon Lee}
\address[]{Jaehoon Lee, Department of Mathematics, The Chinese University of Hong Kong, Shatin, N.T., Hong Kong}
\email{jaehoonlee@cuhk.edu.hk}

\author[E. Yeon]{Eungbeom Yeon}
\address[]{Eungbeom Yeon, Department of Mathematical Sciences, Pusan National University, Busan 46241, Republic of Korea}
\email{ebeom.yeon@pusan.ac.kr}

\allowdisplaybreaks

\begin{document}

\newtheorem{theorem}{theorem}[section]
\newtheorem{thm}[theorem]{Theorem}
\newtheorem{lemma}[theorem]{Lemma}
\newtheorem{cor}[theorem]{Corollary}
\newtheorem{prop}[theorem]{Proposition}
\newtheorem{rmk}[theorem]{Remark}
\newtheorem{Ex}[theorem]{Example}
\newtheorem{conj}[theorem]{Conjecture}
\newtheorem*{mainthm1}{Main Theorem}
\newtheorem*{mainthm2}{Theorem 2}

\numberwithin{equation}{section}
\newcommand{\RNum}[1]{\uppercase\expandafter{\romannumeral #1\relax}}
\newcommand{\R}{\mathbb{R}}
\newcommand{\C}{\mathbb{C}}
\newcommand{\grad}{\nabla}
\newcommand{\laplacian}{\Delta}
\newcommand{\tgamma}{\tilde{\gamma}}
\newcommand{\ttau}{\tilde{\tau}}
\newcommand{\pp}{\Phi}
\newcommand{\tkappa}{\tilde{\kappa}}
\renewcommand{\d}{\textup{d}}
\newlength{\mywidth}
\newcommand\bigfrown[2][\textstyle]{\ensuremath{%
  \array[b]{c}\text{\resizebox{\mywidth}{.7ex}{$#1\frown$}}\\[-1.3ex]#1#2\endarray}}
\newcommand{\arc}[1]{{%
  \setbox9=\hbox{#1}%
  \ooalign{\resizebox{\wd9}{\height}{\texttoptiebar{\phantom{A}}}\cr#1}}}
\newcommand\buildcirclepm[1]{%

  \begin{tikzpicture}[baseline=(X.base), inner sep=-0.1, outer sep=-2]

    \node[draw,circle] (X)  {\footnotesize\raisebox{1ex}{$#1\pm$}};

  \end{tikzpicture}%

}
\newcommand{\introheading}[1]{%
\par\smallskip
\noindent\textbf{#1.}
}


\begin{abstract}
For every positive integer $g$ and every Hoffman--Osserman catenoid in $\mathbb{R}^5$, we construct a complete embedded minimal surface of genus $g$ that is not contained in any proper affine subspace of $\mathbb{R}^5$. These surfaces can be viewed as higher-codimensional analogues of the Costa--Hoffman--Meeks surfaces, associated with the union of the corresponding Hoffman--Osserman catenoid and the plane containing its waist circle. We provide explicit Weierstrass data for these surfaces. Our construction can also be interpreted as restoring conformality to a deformed Costa--Hoffman--Meeks immersion by introducing two additional harmonic coordinate functions.
\end{abstract}

\maketitle

\section{Introduction}\label{intro}
\setcounter{equation}{0}
The Costa surface is a complete embedded minimal surface in $\mathbb{R}^3$ of genus one, with three ends and finite total curvature \cite{C,C2}. Hoffman and Meeks proved its embeddedness and constructed analogous surfaces of every positive genus \cite{HM1}. They also showed that the resulting Costa--Hoffman--Meeks surface of genus $g$ has $4(g+1)$ symmetries and is uniquely characterized among complete embedded minimal surfaces of the same genus with three ends and finite total curvature by having at least $4(g+1)$ symmetries. 

Moreover, these surfaces can be viewed as desingularizations of the union of a catenoid and the plane containing its waist circle. Indeed, Hoffman and Meeks \cite{HM2} showed that, after suitable rigid motions and rescaling to fix the flux of the catenoidal ends, a subsequence converges to this union as the genus tends to infinity. The convergence is smooth on compact subsets away from the intersection circle.  

In this paper, we construct higher-codimensional analogues of the Costa--Hoffman--Meeks surfaces. Motivated by the desingularization picture described above, we seek suitable higher-codimensional counterparts of the catenoid. These arise in the classification of doubly connected minimal surfaces in $\mathbb{R}^n$ with total curvature $-4\pi$ by Hoffman and Osserman \cite[Proposition 6.6]{HO}. Their classification shows that such surfaces are contained in a $5$-dimensional affine subspace and are foliated by ellipses. Thus, each of these surfaces, together with the plane containing one of its ellipses, provides a natural pair for desingularization.

With the $4(g+1)$ symmetries of the Costa--Hoffman--Meeks surfaces in mind, we focus on desingularization pairs that can admit a number of symmetries increasing with the genus. We note that only those arising from the circle-foliated surfaces have this property; see Section \ref{ss313131} for details. Among the surfaces in the Hoffman--Osserman classification, those foliated by circles are the catenoids in $\mathbb{R}^3$, the Lagrangian catenoids in $\mathbb{R}^4$, and the circle-foliated surfaces lying entirely in $\mathbb{R}^5$, which we refer to as the Hoffman--Osserman catenoids. Hence, it is natural to ask whether one can desingularize the union of each of these catenoids and the plane containing its waist circle.

In \cite{JLEY1} and \cite{JLEY2}, the authors studied this desingularization problem for the Lagrangian catenoid in $\mathbb{R}^4$ in genus one and higher genus, respectively.
In the genus-one case, it was proved that such desingularizations must be $J$-holomorphic for some almost complex structure $J$. Moreover, under the additional assumptions of embeddedness and at least $8$ symmetries, it was shown that the number of symmetries must be either $8$ or $12$, with each case uniquely determined up to rigid motions and scaling. In contrast, for higher genus, it was proved that no such embedded desingularization of genus $g$ exists if it has at least $4(g+1)$ symmetries.

We now turn to the Hoffman--Osserman catenoids in $\mathbb{R}^5$ and consider a related existence problem. These catenoids can be parametrized, up to rigid motions and scaling, by
\begin{align*}
\mathsf{HO}(c,f,\alpha):=\left\{\textup{Re} \left( - \frac{\alpha^2}{4c} \xi - \frac{c}{\xi}, i \frac{\alpha^2}{4c} \xi - \frac{ic}{\xi},  \alpha \log \xi,  f \xi, i f \xi \right) \in\mathbb{R}^5\mid \xi\in\mathbb{C}\setminus\{0\}\right\},
\end{align*}
where $\alpha$ is an arbitrary positive real number and $c,f$ are nonzero complex numbers. The level curves $|\xi|=\mathrm{constant}$ are mapped to circles. In particular, the waist circle corresponds to
\begin{align*}
|\xi|=\left( \frac{16 |c|^4}{16|c|^2 |f|^2 + \alpha^4} \right)^{\frac{1}{4}}.
\end{align*}
We denote the plane containing this waist circle by $\mathsf{P}(c,f,\alpha)$.

Our main result is the following.
\begin{mainthm1}
Let $\mathsf{HO}(c,f,\alpha)$ and $\mathsf{P}(c,f,\alpha)$ denote the Hoffman--Osserman catenoid and the plane containing its waist circle, respectively, as defined above. For each integer $g\geq 1$, there exists a complete embedded minimal surface $\mathsf{S_g}(c,f,\alpha)$ of genus $g$ in $\mathbb{R}^5$ with the following properties:
\begin{enumerate}
\item $\mathsf{S_g}(c,f,\alpha)$ is not contained in any proper affine subspace of $\mathbb{R}^5$;
\item The total curvature of $\mathsf{S_g}(c,f,\alpha)$ is $-4\pi(g+2)$;
\item $\mathsf{S_g}(c,f,\alpha)$ has two embedded ends with nonzero logarithmic growth and one embedded planar end; 
\item The limiting tangent planes of the two ends with nonzero logarithmic growth coincide with those of the two ends of $\mathsf{HO}(c,f,\alpha)$, while the limiting tangent plane of the planar end is parallel to $\mathsf{P}(c,f,\alpha)$;
\item $\mathsf{S_g}(c,f,\alpha)$ admits a dihedral symmetry group of order $4(g+1)$, generated by
\begin{align*}
\Lambda :=
\begin{pmatrix}
\mathcal{S}_{\frac{\pi}{g+1}} & 
\\ & 1 &
\\ & &\mathcal{S}_{-\frac{\pi}{g+1}}
\end{pmatrix}\cdot U,
\qquad
K := 
\begin{pmatrix}
\mathcal{T}_{2\theta_c}  & 
\\ & 1 &
\\ & &\mathcal{T}_{2\theta_f} 
\end{pmatrix},
\end{align*}
where $U$ denotes the reflection across the plane passing through the origin and parallel to $\mathsf{P}(c,f,\alpha)$,
\begin{align*}
c=|c|e^{i\theta_c},\quad f=|f|e^{i\theta_f},
\end{align*}
and 
\begin{align*} 
\mathcal{S}_{\theta}:=
\begin{pmatrix}
\cos\theta & -\sin\theta\\
\sin\theta & \cos\theta
\end{pmatrix},\quad 
\mathcal{T}_{\theta}:=
\begin{pmatrix}
\cos\theta & -\sin\theta\\
-\sin\theta & -\cos\theta
\end{pmatrix},
\quad \theta\in\mathbb{R};
\end{align*}
\item $\mathsf{S_g}(c,f,\alpha)$ can be represented as a graph over a deformed Costa--Hoffman--Meeks surface of genus $g$;
\item $\mathsf{S_g}(c,f,\alpha)$ is conformally equivalent to the Costa--Hoffman--Meeks surface of genus $g$.
\end{enumerate}
\end{mainthm1}
We mention that none of these surfaces contains a straight line; see Remark \ref{rmk57}. 

\introheading{Construction strategy} 
Rather than addressing the desingularization problem directly, we take a different approach based on the generalized Weierstrass representation. The geometric properties expected of such a desingularization can be translated into conditions on the Weierstrass data, and our goal is to construct well-defined data satisfying these conditions.

In $\mathbb{R}^3$, embeddedness and at least $4(g+1)$ symmetries, in the setting of the Costa--Hoffman--Meeks characterization, determine the underlying conformal structure and the orientations of the ends \cite{HM1}. In higher codimension, however, additional possibilities arise. In principle, one can describe the Weierstrass data corresponding to these possibilities using various parameters, but the resulting period problems can be highly complicated. Instead of investigating all such possibilities, we make several geometric choices to establish the existence of the desired surfaces.

More precisely, we choose the underlying Riemann surface to be conformally equivalent to the Costa--Hoffman--Meeks surface of the same genus and require the intrinsic symmetries induced by ambient isometries to agree with those of the Costa--Hoffman--Meeks surface. Moreover, there are several possible realizations of the ambient symmetries, depending, for instance, on the choice of the reflection $U$ in the theorem. We restrict ourselves to the particular choice of $U$ specified there and construct well-defined Weierstrass data satisfying the corresponding conditions. The explicit Weierstrass data for our construction and the underlying ideas are presented in Section \ref{WAWAWA}.

These choices leave open the possibility of other types of minimal surfaces in higher codimension. For comparison, in the genus-one case in $\mathbb{R}^4$, the desingularization with $8$ symmetries is conformally equivalent to the Costa surface, while a distinct surface with $12$ symmetries also exists. This raises the question of whether there are other desingularizations associated with the Hoffman--Osserman catenoids that are not conformally equivalent to the Costa--Hoffman--Meeks surfaces and have at least $4(g+1)$ symmetries.

\introheading{An alternative viewpoint}
Our construction also admits another interpretation. As indicated in Property $(6)$ of the theorem, the resulting surfaces can be represented as graphs over suitably deformed Costa--Hoffman--Meeks surfaces. Such a deformation preserves harmonicity with respect to the original conformal structure but breaks conformality. From this perspective, the construction can be understood as restoring conformality by introducing two additional coordinate functions.

A similar idea appears in a general result of Alarc\'on, Fern\'andez, and L\'opez \cite[Theorem~B]{AFL}, which asserts that any nonconstant harmonic map from an open Riemann surface into $\mathbb{R}^{N-2}$, $N\geq 3$, admits a complete conformal minimal lift into $\mathbb{R}^{N}$ obtained by adding two harmonic coordinate functions. Our construction realizes such a completion with precise control over the geometry of the resulting surfaces. A key feature is the use of explicit meromorphic data with carefully chosen poles and residues. These allow us to control the asymptotic behavior of the ends, satisfy the period conditions, and realize the prescribed end geometry and symmetries.

\introheading{Blow-up limits}
The limiting behavior of the surfaces constructed in this paper suggests further questions. The classical Costa--Hoffman--Meeks surfaces can be translated and rescaled so that the maximum absolute Gaussian curvature is $1$ and is attained at the origin. A subsequence then converges smoothly on compact subsets to the singly periodic Scherk surface \cite{HM2}. We are currently investigating analogous blow-up limits near the handles of these surfaces as the genus tends to infinity, including how these limits depend on the deformation parameters.

In particular, it remains to determine whether the limiting surfaces are classical Scherk surfaces contained in three-dimensional affine subspaces or whether other models can arise in higher codimension. Such an analysis may provide local models for future gluing constructions, potentially leading to further examples of minimal surfaces in higher codimension.

The paper is organized as follows. Section \ref{pre} reviews the generalized Weierstrass representation and discusses the form of the Weierstrass data corresponding to multiplicity-one embedded ends. In Section \ref{sam33}, we investigate the number of symmetries associated with the surfaces appearing in the Hoffman--Osserman classification and show that the circle-foliated Hoffman--Osserman catenoids are suitable for our construction. We also determine the possible forms of their symmetries. Section \ref{WAWAWA} is devoted to presenting explicit Weierstrass data for surfaces satisfying the properties stated in the main theorem and discussing the ideas behind their construction. Finally, we prove the main theorem in Section \ref{main} by showing that these Weierstrass data give rise to well-defined minimal surfaces satisfying all the required properties.


\section*{Acknowledgements}
JL was supported by a KIAS Individual Grant MG086402 at Korea Institute for Advanced Study, in part by National Research Foundation of Korea [No. RS-2026-25472185], and substantially by research grants from the Research Grants Council of the Hong Kong Special Administrative Region, China [Project No.: CUHK 14306624]. EY was supported by National Research Foundation of Korea NRF-2022R1C1C2013384. 

\section{Preliminaries}\label{pre}
\setcounter{equation}{0}
\subsection{Generalized Weierstrass representation}
The generalized Weierstrass representation is a powerful tool for studying minimal surfaces in $\mathbb{R}^n$. It is particularly useful for surfaces of finite total curvature with prescribed end behavior, where it can be combined with Riemann surface theory. Since this representation plays a central role in our construction, we briefly recall it here for completeness. For further details on the generalized Weierstrass representation and its applications to minimal surfaces in $\mathbb{R}^4$ and higher-dimensional Euclidean spaces, we refer the reader to \cite{HO,JLEY1,JLEY2}.

Every complete immersed minimal surface in $\mathbb{R}^N$ of finite total curvature and genus $g$ admits a parametrization $X: \Sigma_g\setminus\{q_1,q_2,\cdots,q_k\}\to\mathbb{R}^N$, where $\Sigma_g$ is a closed Riemann surface of genus $g$ and $q_1,q_2,\cdots,q_k$ are distinct points, given by
\begin{align}\label{gwr}
X=\text{Re}\int (\phi_1,\phi_2,\cdots,\phi_N)+b
\end{align}
for some $b\in\mathbb{R}^N$. Here, the $\phi_j$ are meromorphic $1$-forms on $\Sigma_g$ satisfying the following conditions:
\begin{itemize}
\item[(1)] Each $\phi_j$ is holomorphic on $\Sigma_g\setminus\{q_1,q_2,\cdots,q_k\}$, and at least one of the $\phi_j$ has a pole at each puncture.
\item[(2)] $\sum_{j=1}^N\phi_j^2\equiv0$ and $\sum_{j=1}^N|\phi_j|^2>0$.
\item[(3)] Each $\phi_j$ has vanishing real periods on $\Sigma_g\setminus\{q_1,q_2,\cdots,q_k\}$.
\end{itemize}
Conversely, if a collection of meromorphic $1$-forms satisfies the above conditions, the map $X$ given by (\ref{gwr}) is a conformal harmonic immersion, giving rise to a complete minimal surface of finite total curvature and genus $g$ in $\mathbb{R}^N$. Each puncture $q_j$ corresponds to an end of the surface. 

The formula (\ref{gwr}) is called the \textbf{generalized Weierstrass representation}, and the meromorphic $1$-forms $\phi_j$ are called the \emph{Weierstrass data}. The map 
\begin{align*}
[\phi_1,\phi_2,\cdots,\phi_N]:\Sigma_g\to\mathbb{P}^{N-1}
\end{align*}
is called the generalized Gauss map. By the conformality condition $(2)$, the image of the generalized Gauss map lies in the quadric 
\begin{align*}
\left\{[\zeta_1,\zeta_2,\cdots,\zeta_N]\in\mathbb{P}^{N-1}\ |\ \sum_{j=1}^N\zeta_j^2=0\right\}.
\end{align*}

\subsection{Embedded end of multiplicity one}\label{emenmulone}
In higher codimensions, embeddedness alone does not necessarily rule out minimal ends with higher multiplicity. Since we are particularly interested in multiplicity-one ends, such as those of the Lagrangian catenoids and Hoffman--Osserman catenoids, we restrict our attention to meromorphic $1$-forms belonging to
\begin{align*}
H^0(\Sigma_g, C\otimes [2q_1+2q_2+\cdots+2q_k]),
\end{align*}
where $C$ denotes the canonical line bundle of $\Sigma_g$.

We first recall the definition of a multiplicity-one embedded planar end. A complete minimal immersion $X:\Sigma_g\setminus\{q_1,q_2,\cdots,q_k\}\to\mathbb{R}^N$ of finite total curvature is said to have such an end at $q_j$ if, outside a compact subset, the end is asymptotic to a plane and can be represented as a graph over that plane, with no logarithmic growth in any coordinate function. This definition admits an equivalent characterization in terms of the Weierstrass data. For further discussion of this characterization, see \cite{JLEY1,JLEY2}.

\begin{prop}
A minimal immersion $X$ given by the generalized Weierstrass representation (\ref{gwr}) has an embedded planar end at $q\in\{q_1,q_2,\cdots,q_k\}$ if and only if
\begin{align}\label{OrdRes}
\min_{1\leq j\leq N}\textup{ord}_{q}\phi_j=-2\quad \text{and}\quad \textup{res}_{q}\phi_1=\textup{res}_{q}\phi_2=\cdots=\textup{res}_{q}\phi_N=0.
\end{align} 
\end{prop}

Thus, if the Weierstrass data in $H^0(\Sigma_g, C\otimes [2q_1+2q_2+\cdots+2q_k])$ satisfy condition (\ref{OrdRes}) at a puncture, the corresponding end is an embedded planar end. Otherwise, the end may have nonzero logarithmic growth.

For each puncture $q_j$, choose
\begin{align}\label{shshsh}
\eta_j\in H^0(\Sigma_g,C\otimes [2q_j])\setminus H^0(\Sigma_g,C).
\end{align}
Since $\eta_j$ has a pole only at $q_j$, the residue formula implies that $\eta_j$ has vanishing residues and
\begin{align*}
\textup{ord}_{q_j}\eta_j=-2.
\end{align*}
Additionally, for $2\leq j\leq k$, choose
\begin{align}\label{hshshs}
\eta_{1j}\in H^0(\Sigma_g,C\otimes [2q_1+2q_j])\setminus\left(H^0(\Sigma_g,C)\oplus\mathbb{C}\eta_1\oplus\mathbb{C}\eta_j\right). 
\end{align}
By the definition of $\eta_{1j}$ and the residue formula, we have
\begin{align*}
\textup{res}_{q_1}\eta_{1j}=-\textup{res}_{q_j}\eta_{1j}\neq 0.
\end{align*}
Using the Riemann--Roch formula, $H^0(\Sigma_g, C\otimes [2q_1+2q_2+\cdots+2q_k])$ has dimension $g+2k-1$. Hence, it can be decomposed as
\begin{align*}
H^0(\Sigma_g, C)\oplus\left(\bigoplus_{1\leq j\leq k}\mathbb{C}\eta_j\right)\oplus\left(\bigoplus_{2\leq j\leq k}\mathbb{C}\eta_{1j}\right).
\end{align*}

We now specialize to the case of two ends with nonzero logarithmic growth and one embedded planar end, which is the main focus of this paper. Suppose that the Weierstrass data define a minimal immersion with ends of nonzero logarithmic growth at $q_1$ and $q_3$ and an embedded planar end at $q_2$. Then the residues at $q_2$ must vanish, forcing the coefficient of $\eta_{12}$ in each component to be zero. Consequently, each component of the Weierstrass data belongs to 
\begin{align}\label{guide}
H^0(\Sigma_g,C)\oplus\mathbb{C}\eta_1\oplus\mathbb{C}\eta_2\oplus\mathbb{C}\eta_3\oplus\mathbb{C}\eta_{13}.
\end{align}
This observation will guide our construction of the Weierstrass data.
\section{The Hoffman--Osserman Classification and Symmetries}\label{sam33}
\setcounter{equation}{0}
We first recall the classification of doubly connected minimal surfaces in $\mathbb{R}^n$ with total curvature $-4\pi$ by Hoffman and Osserman.

\begin{prop}[Proposition $6.6$ in \cite{HO}]
Let $S$ be a doubly connected minimal surface in $\mathbb{R}^n$ with total curvature $-4\pi$. If $S$ is complete and regular, then it lies in some $5$-dimensional affine subspace and is of one of the two following forms:
\begin{enumerate}
\item $S$ is given by
\begin{align*}
X = \textup{Re} \left( d_1 \xi - \frac{c}{\xi},  d_2 \xi - \frac{ic}{\xi},  \alpha \log \xi,  d_4 \xi, d_5 \xi \right)
\end{align*}
in $0 < |\xi| < \infty$, where $c$ is a nonzero complex constant, $d_4,d_5$ are arbitrary complex constants, $\alpha$ is an arbitrary positive real constant, and $d_1, d_2$ are defined by
\begin{align*}
d_1 = \frac{c}{\alpha^2}\left(d_4^2 + d_5^2\right) - \frac{\alpha^2}{4c},\quad d_2 = i\left(\frac{c}{\alpha^2}\left(d_4^2 + d_5^2\right) + \frac{\alpha^2}{4c} \right).
\end{align*}

\item $S$ lies in a $4$-dimensional subspace and in terms of suitable coordinates, $S$ is the graph in $\mathbb{C}^2$ of the function
\begin{align*}
w = \alpha z + \frac{b}{z}, \quad b \neq 0, \quad 0 < |z| < \infty.
\end{align*}
\end{enumerate}
\end{prop}

As noted in \cite{HO}, all surfaces appearing in the above classification are foliated by ellipses given by $|\xi|=\mathrm{constant}$. Moreover, Case $(1)$ is further divided into four subcases:
\begin{enumerate}
\item[(1.a)] If $d_4=d_5=0$, then $S$ is a catenoid in $\mathbb{R}^3$.
\item[(1.b)] If $d_4,d_5$ are not both zero but are linearly dependent over $\mathbb{R}$, then $S$ is contained in a $4$-dimensional affine subspace.
\item[(1.c)] If $d_4,d_5$ are not both zero and $d_4^2+d_5^2=0$, then $S$ is not contained in any proper affine subspace of $\mathbb{R}^5$.
\item[(1.d)] If $d_4,d_5$ are linearly independent over $\mathbb{R}$ and $d_4^2+d_5^2\neq 0$, then $S$ is not contained in any proper affine subspace of $\mathbb{R}^5$.
\end{enumerate}
We remark that the circle-foliated surfaces are precisely those in Cases (1.a) and (1.c), together with those in Case $(2)$ with $\alpha=0$.

\begin{rmk}
\normalfont
There is a minor typo in the statement of Proposition $6.6$ in \cite{HO}, so we present the corrected version here. The corresponding expression, however, appears correctly in its proof. 
\end{rmk}
\begin{rmk}
\normalfont
The surfaces in Case $(1)$ depend on the parameters $\alpha, c, d_4, d_5$. We note that the parameter choices $(\alpha, c, d_4,d_5)$ and $(\alpha',c',d_4',d_5')$ yield the same surface up to rigid motions and scaling if and only if
\begin{align*}
\frac{|c|^2}{\alpha^4}\left|d_4^2+d_5^2\right|=\frac{|c'|^2}{\alpha'^4}\left|d_4'^2+d_5'^2\right|,\quad \frac{|c|^2}{\alpha^4}\left(|d_4|^2+|d_5|^2\right)=\frac{|c'|^2}{\alpha'^4}\left(|d_4'|^2+|d_5'|^2\right).
\end{align*}
Consequently, these surfaces form a two-parameter family and admit a simpler representation. Nevertheless, we retain the notation of \cite{HO} for consistency and to allow greater flexibility in choosing parameters for later applications.
\end{rmk}

\subsection{Number of symmetries}\label{ss313131}
Since the surfaces in Case $(2)$ and the symmetries associated with them have been studied in \cite{JLEY2}, we now consider the surfaces in Case $(1)$. Their Weierstrass data are given by
\begin{align*}
\left(\phi_1,\phi_2,\phi_3,\phi_4, \phi_5 \right)=2X_{\xi}d\xi=\left(\left(d_1+\frac{c}{\xi^2}\right)d\xi, \left(d_2+\frac{ic}{\xi^2}\right)d\xi, \frac{\alpha}{\xi}d\xi, d_4d\xi, d_5d\xi\right),
\end{align*}
and the generalized Gauss map $\Phi:\mathbb{C}\setminus\{0\}\to\mathbb{P}^4$ is given by
\begin{align*}
\Phi (\xi)=\left[ d_1+\frac{c}{\xi^2}, d_2+\frac{ic}{\xi^2}, \frac{\alpha}{\xi}, d_4, d_5\right]\in\mathbb{P}^4.
\end{align*}
In particular, the limiting values of $\Phi$ at $\xi=0$ and $\xi=\infty$ are
\begin{align*}
[1,i,0,0,0]\quad\text{and}\quad \left[d_1,d_2,0,d_4,d_5\right],
\end{align*}
respectively.

For a surface $S\subset\mathbb{R}^5$, a symmetry of $S$ is a rigid motion $R=(A,b)$ of $\mathbb{R}^5$ satisfying $R(S)=S$, where $A\in O(5)$ and $b\in\mathbb{R}^5$. Here $R$ acts on $\mathbb{R}^5$ by $R(x)=Ax+b$. The symmetries of $S$ form a group. Furthermore, if $S$ admits no nontrivial translational symmetry, then each symmetry is uniquely determined by its linear part.

In what follows, we focus on the surfaces in Case (1.d), which are not contained in any proper affine subspace of $\mathbb{R}^5$ and are foliated by ellipses rather than circles. We show that the number of symmetries of their possible desingularizations is bounded. This rules out Case (1.d) for our construction, since we seek surfaces whose number of symmetries increases with genus, as in the Costa--Hoffman--Meeks surfaces. The surfaces in Case (1.b), which lie in a four-dimensional affine subspace, can be ruled out using a similar argument in $\mathbb{R}^4$ after a suitable rigid motion. 

Suppose that there exists a minimal surface $S\subset\mathbb{R}^5$ that desingularizes the union of a surface in Case (1.d) and the plane containing one of its ellipses. Any symmetry of $S$ must either preserve or interchange its two ends of nonzero logarithmic growth. Hence, the subgroup of symmetries preserving each end has index at most two in the full symmetry group. For any such symmetry $R=(A,b)$, its linear part $A$ must preserve the limiting tangent planes at these two ends.

To preserve the limiting tangent plane corresponding to $[1,i,0,0,0]\in\mathbb{P}^4$, $A$ must have the form
\begin{align*}
A=
\begin{pmatrix}
O_1 & \\
& O_2
\end{pmatrix}
\end{align*}
with respect to the standard basis of $\mathbb{R}^5$, where $O_1 \in O(2)$ and $O_2 \in O(3)$. Since $d_1^2+d_2^2+d_4^2+d_5^2=0$, the two vectors
\begin{align*}
\left(\textup{Re}\ d_1,\textup{Re}\ d_2,0,\textup{Re}\ d_4,\textup{Re}\ d_5\right),\quad\left(\textup{Im}\ d_1,\textup{Im}\ d_2,0,\textup{Im}\ d_4,\textup{Im}\ d_5\right) 
\end{align*}
are orthogonal and have the same norm. Thus, the invariance of the limiting tangent plane corresponding to $\left[d_1,d_2,0,d_4,d_5\right]\in\mathbb{P}^4$ implies that
\begin{align}\label{Aeqeqeq}
A\begin{pmatrix}
\textup{Re}\ d_1 & \textup{Im}\ d_1\\
\textup{Re}\ d_2 & \textup{Im}\ d_2\\
0 & 0\\
\textup{Re}\ d_4 & \textup{Im}\ d_4\\
\textup{Re}\ d_5 & \textup{Im}\ d_5
\end{pmatrix}=\begin{pmatrix}
\textup{Re}\ d_1 & \textup{Im}\ d_1\\
\textup{Re}\ d_2 & \textup{Im}\ d_2\\
0 & 0\\
\textup{Re}\ d_4 & \textup{Im}\ d_4\\
\textup{Re}\ d_5 & \textup{Im}\ d_5
\end{pmatrix}\begin{pmatrix}
\cos\eta & -\sin\eta\\
\epsilon\sin\eta & \epsilon\cos\eta
\end{pmatrix}
\end{align}
for some $\eta\in[0,2\pi)$ and $\epsilon\in\{-1,+1\}$. In particular,
\begin{align}\label{O1EQO1EQ}
O_1\begin{pmatrix}
\textup{Re}\ d_1 & \textup{Im}\ d_1\\
\textup{Re}\ d_2 & \textup{Im}\ d_2
\end{pmatrix}=\begin{pmatrix}
\textup{Re}\ d_1 & \textup{Im}\ d_1\\
\textup{Re}\ d_2 & \textup{Im}\ d_2
\end{pmatrix}\begin{pmatrix}
\cos\eta & -\sin\eta\\
\epsilon\sin\eta & \epsilon\cos\eta
\end{pmatrix}
\end{align}
and
\begin{align}\label{O2EQO2EQ}
O_2\begin{pmatrix}
0 & 0\\
\textup{Re}\ d_4 & \textup{Im}\ d_4\\
\textup{Re}\ d_5 & \textup{Im}\ d_5
\end{pmatrix}=\begin{pmatrix}
0 & 0\\
\textup{Re}\ d_4 & \textup{Im}\ d_4\\
\textup{Re}\ d_5 & \textup{Im}\ d_5
\end{pmatrix}\begin{pmatrix}
\cos\eta & -\sin\eta\\
\epsilon\sin\eta & \epsilon\cos\eta
\end{pmatrix}.
\end{align}

Comparing the norms of the corresponding column vectors on both sides of each matrix equation, we obtain the same relation
\begin{align*}
\sin^2\eta\ \textup{Re}\left(d_1^2+d_2^2\right)=\epsilon\cos\eta\sin\eta\ \textup{Im}\left(d_1^2+d_2^2\right).
\end{align*}
In Case (1.d), we have $d_1^2+d_2^2\neq 0$. Thus, for each fixed $\epsilon\in\{-1,+1\}$, there are at most four values of $\eta\in[0,2\pi)$ satisfying this equation. Moreover, since $d_4$ and $d_5$ are linearly independent over $\mathbb{R}$, the vectors $\left(0, \textup{Re}\ d_4, \textup{Re}\ d_5\right)$ and $\left(0, \textup{Im}\ d_4, \textup{Im}\ d_5\right)$ are linearly independent. Hence, the number of possible matrices $O_2$ satisfying (\ref{O2EQO2EQ}) is bounded. 

On the other hand, since the matrix 
\begin{align*}
\begin{pmatrix}
\textup{Re}\ d_1 & \textup{Im}\ d_1\\
\textup{Re}\ d_2 & \textup{Im}\ d_2
\end{pmatrix}
\end{align*}
is nonzero in (\ref{O1EQO1EQ}), there are at most two possible matrices $O_1\in O(2)$ satisfying (\ref{O1EQO1EQ}) for each fixed $\eta$ and $\epsilon$. Consequently, the number of possible matrices $A$ is bounded. Since the subgroup of end-preserving symmetries has index at most two, the order of the full symmetry group is also bounded in Case (1.d). 

Thus, the only remaining possibilities for constructing surfaces whose number of symmetries increases with genus are Cases (1.a) and (1.c), which are precisely the circle-foliated surfaces in Case $(1)$.

\subsection{The Hoffman--Osserman catenoids $\mathsf{HO}(c,f,\alpha)$}\label{dubidubab}
We now consider the circle-foliated surfaces in Case (1.c). Up to a reflection, we may assume that 
\begin{align*}
d_4=f,\quad d_5=if
\end{align*}
for some $f\in\mathbb{C}\setminus\{0\}$. The surface is then parametrized by the immersion $X:\mathbb{C}\setminus\{0\}\to\mathbb{R}^5$ given by
\begin{align*}
X(\xi) = \textup{Re} \left( - \frac{\alpha^2}{4c} \xi - \frac{c}{\xi}, i \frac{\alpha^2}{4c} \xi - \frac{ic}{\xi},  \alpha \log \xi,  f \xi, i f \xi \right) \in\mathbb{R}^5,
\end{align*}
where $\alpha>0$ and $c,f\in\mathbb{C}\setminus\{0\}$. We call such a surface a Hoffman--Osserman catenoid and denote it by $\mathsf{HO}(c,f,\alpha)$. The limiting values of its generalized Gauss map at the two ends $\xi=0$ and $\xi=\infty$ are 
\begin{align*}
\left[ 1,i,0,0,0 \right] \quad \text{and} \quad \left[ - \frac{\alpha^2}{4c} , i \frac{\alpha^2}{4c},0,f,if \right],
\end{align*}
respectively.

Setting $\xi=re^{i\theta}$, we obtain
\begin{align*}
X(\xi)=\left(0, 0, \alpha\log r,0,0\right)+\cos\theta F_1(r)+\sin\theta F_2(r),
\end{align*}
where $F_1(r)$ and $F_2(r)$ are the orthogonal vectors of the same norm given by
\begin{align*}
F_1(r)&=\left(\frac{1}{r}+\frac{\alpha^2}{4|c|^2}r\right)\left(-\textup{Re}\ c,\textup{Im}\ c,0,0,0\right)+r\left(0,0,0,\textup{Re}\ f,-\textup{Im}\ f\right),\\
F_2(r)&=-\left(\frac{1}{r}+\frac{\alpha^2}{4|c|^2}r\right)\left(\textup{Im}\ c,\textup{Re}\ c,0,0,0\right)-r\left(0,0,0,\textup{Im}\ f,\textup{Re}\ f\right).
\end{align*}
When the parameters are clear from the context, we suppress the dependence on $c,f,\alpha$ in the notation.
The circles attain their minimum radius when
\begin{align*}
|\xi|=\left(\frac{16|c|^4}{16|c|^2|f|^2+\alpha^4}\right)^{\frac{1}{4}}=:r_{waist}.
\end{align*}
We denote the plane containing this waist circle by $\mathsf{P}(c,f,\alpha)$.

Suppose that there exists a minimal surface $S(c,f,\alpha)\subset\mathbb{R}^5$ with two ends of nonzero logarithmic growth whose limiting tangent planes coincide with those of $\mathsf{HO}(c,f,\alpha)$, and one planar end parallel to $\mathsf{P}(c,f,\alpha)$. We now determine the form of symmetries that preserve each end of $S(c,f,\alpha)$. 

Let $(A,b)$ be an end-preserving symmetry of $S(c,f,\alpha)$. By an argument similar to that in the previous subsection, $A$ has the form
\begin{align*}
A=\begin{pmatrix}
O_1 & \\
& O_2
\end{pmatrix}
\end{align*}
for some $O_1\in O(2)$ and $O_2\in O(3)$, and satisfies (\ref{Aeqeqeq}) for some $\eta\in[0,2\pi)$ and $\epsilon\in\{-1,+1\}$. Since
\begin{align*}
[d_1,d_2,0,d_4,d_5]=\left[ - \frac{\alpha^2}{4c} , i \frac{\alpha^2}{4c},0,f,if \right],
\end{align*}
equation (\ref{Aeqeqeq}) yields
\begin{align}\label{O1O1O1}
O_1\begin{pmatrix}
-\cos\theta_c & \sin\theta_c\\
\sin\theta_c & \cos\theta_c
\end{pmatrix}=\begin{pmatrix}
-\cos\theta_c & \sin\theta_c\\
\sin\theta_c & \cos\theta_c
\end{pmatrix}\begin{pmatrix}
\cos\eta & -\sin\eta\\
\epsilon\sin\eta & \epsilon\cos\eta
\end{pmatrix}
\end{align}
and
\begin{align}\label{O2O2O2}
O_2\begin{pmatrix}
0 & 0\\
\cos\theta_f & \sin\theta_f\\
-\sin\theta_f & \cos\theta_f
\end{pmatrix}=\begin{pmatrix}
0 & 0\\
\cos\theta_f & \sin\theta_f\\
-\sin\theta_f & \cos\theta_f
\end{pmatrix}\begin{pmatrix}
\cos\eta & -\sin\eta\\
\epsilon\sin\eta & \epsilon\cos\eta
\end{pmatrix},
\end{align}
where $c=|c|e^{i\theta_c}$ and $f=|f|e^{i\theta_f}$. It follows from (\ref{O2O2O2}) that
\begin{align*}
O_2=\begin{pmatrix}
\pm 1 & 0\\
0 & \widetilde{O}_2
\end{pmatrix},
\end{align*}
where $\widetilde{O}_2\in O(2)$. Substituting this back into the same equation, we obtain
\begin{align}\label{HAHAHA}
\widetilde{O}_2\begin{pmatrix}
\cos\theta_f & \sin\theta_f\\
-\sin\theta_f & \cos\theta_f
\end{pmatrix}=\begin{pmatrix}
\cos\theta_f & \sin\theta_f\\
-\sin\theta_f & \cos\theta_f
\end{pmatrix}\begin{pmatrix}
\cos\eta & -\sin\eta\\
\epsilon\sin\eta & \epsilon\cos\eta
\end{pmatrix}.
\end{align}

Recall that every $2\times 2$ real orthogonal matrix is one of the following forms:
\begin{align*} 
\mathcal{S}_{\theta}:=
\begin{pmatrix}
\cos\theta & -\sin\theta\\
\sin\theta & \cos\theta
\end{pmatrix},\quad 
\mathcal{T}_{\theta}:=
\begin{pmatrix}
\cos\theta & -\sin\theta\\
-\sin\theta & -\cos\theta
\end{pmatrix},
\quad \theta\in\mathbb{R}.
\end{align*}
These matrices satisfy the relations
\begin{align*}
\mathcal{S}_{\theta}\mathcal{T}_{\phi}=\mathcal{T}_{\phi-\theta},\quad \mathcal{T}_{\phi}\mathcal{S}_{\theta}=\mathcal{T}_{\phi+\theta},\quad \mathcal{S}_{\theta}\mathcal{S}_{\phi}=\mathcal{S}_{\theta+\phi},\quad \mathcal{T}_{\theta}\mathcal{T}_{\phi}=\mathcal{S}_{\phi-\theta}.
\end{align*}
If $\epsilon=1$, then (\ref{O1O1O1}) and (\ref{HAHAHA}) yield
\begin{align*}
O_1\mathcal{T}_{\theta_c}=\mathcal{T}_{\theta_c}\mathcal{S}_{\eta},\quad \widetilde{O}_2\mathcal{S}_{-\theta_f}=\mathcal{S}_{-\theta_f}\mathcal{S}_{\eta},
\end{align*}
and hence
\begin{align*}
O_1=\mathcal{S}_{-\eta},\quad \widetilde{O}_2=\mathcal{S}_{\eta}.
\end{align*}
Similarly, if $\epsilon=-1$, then (\ref{O1O1O1}) and (\ref{HAHAHA}) yield
\begin{align*}
O_1=\mathcal{T}_{2\theta_c-\eta},\quad \widetilde{O}_2=\mathcal{T}_{2\theta_f+\eta}.
\end{align*}
Consequently, $A$ takes one of the following forms for some $\eta\in[0,2\pi)$:
\begin{align}\label{aa}
\begin{pmatrix}
\mathcal{S}_{-\eta} &  \\
&\pm1& \\
 && \mathcal{S}_{\eta}
\end{pmatrix}
\quad\text{or}\quad
\begin{pmatrix}
\mathcal{T}_{2\theta_c-\eta} & & \\
& \pm 1 & \\
& & \mathcal{T}_{2\theta_f+\eta}
\end{pmatrix}.
\end{align}

A direct computation shows that these orthogonal matrices also preserve the plane spanned by $F_1(r_{waist})$ and $F_2(r_{waist})$, which is parallel to $\mathsf{P}(c,f,\alpha)$. Thus, the matrices in (\ref{aa}) give the possible forms of the linear parts of end-preserving symmetries of $S(c,f,\alpha)$.

\section{Weierstrass Data for $\mathsf{S_g}(c,f,\alpha)$}\label{WAWAWA}
\setcounter{equation}{0}
In this section, we present explicit Weierstrass data for the surfaces $\mathsf{S_g}(c,f,\alpha)$ in the main theorem. 
\subsection{Explicit Weierstrass data}\label{GUGUGU}
Let $\Sigma_g$ be the Riemann surface defined by
\begin{align*}
\Sigma_g=\left\{(z, w)\in\left(\mathbb{C}\cup\{\infty\}\right)\times\left(\mathbb{C}\cup\{\infty\}\right)\ \middle|\ w^{g+1}=z^g(z+1)(z-1)\right\}.
\end{align*}
We denote three points on $\Sigma_g$ by
\begin{align*}
q_1=(1,0),\quad q_2=\infty,\quad q_3=(-1,0).
\end{align*}
The Weierstrass data will be chosen so that $q_1$ and $q_3$ correspond to the two ends of nonzero logarithmic growth, while $q_2$ corresponds to the planar end.
Define
\begin{align*}
\eta_1:=\frac{dz}{w(z-1)},\quad \eta_2:=\frac{wdz}{(z+1)(z-1)},\quad \eta_3:=\frac{dz}{w(z+1)},
\end{align*}
and
\begin{align*}
\eta_{13}:=\frac{dz}{(z+1)(z-1)}
\end{align*}
on $\Sigma_g$. These meromorphic $1$-forms satisfy (\ref{shshsh}) and (\ref{hshshs}).

We introduce the constants
\begin{align*}
E_0&:=\int_0^1 \frac{\bigl(t^g(1-t^2)\bigr)^{\frac{1}{g+1}}}{t(1-t^2)}\,\textup{d}t ,
\quad
E_1
:=
\int_0^1
\frac{t}{\bigl(t^g(1-t^2)\bigr)^{\frac{1}{g+1}}(t+1)}\,\textup{d}t,\\
E_2&:=
\int_0^1
\frac{\bigl(t^g(1-t^2)\bigr)^{\frac{1}{g+1}}}{1-t^2}\,\textup{d}t ,
\quad
E_3:=
\int_0^1
\frac{1}{\bigl(t^g(1-t^2)\bigr)^{\frac{1}{g+1}}(t+1)}\,\textup{d}t .
\end{align*}
For each $\alpha>0$ and $c,f\in\mathbb{C}\setminus\{0\}$, define
\begin{align*}
B(c,f,\alpha):=\frac{E_1+E_3}{2|c|^2E_2}r_{waist}^2,\quad H(c,f,\alpha):=\frac{E_3-E_1}{2\overline{c}E_0}\left(1-\frac{\alpha^2}{4|c|^2}r_{waist}^2\right),
\end{align*}
where
\begin{align*}
r_{waist}=\left(\frac{16|c|^4}{16|c|^2|f|^2+\alpha^4}\right)^{\frac{1}{4}}.
\end{align*}
We also define $L(c,f,\alpha)>0$ by
\begin{align*}
L(c,f,\alpha)^2 = 8B(c,f,\alpha)\left( \frac{1}{r_{waist}^2} + \frac{\alpha^2}{4|c|^2} \right) - \frac{8H(c,f,\alpha)}{c}.
\end{align*}

We note that $L(c,f,\alpha)$ is well-defined. Indeed, substituting the expressions for $B$ and $H$ into the right-hand side yields
\begin{align*}
\frac{4}{|c|^2}\left[\left(\frac{E_1+E_3}{E_2}+\frac{E_3-E_1}{E_0}\right)\frac{\alpha^2 }{4|c|^2}r_{waist}^2+\left(\frac{E_1+E_3}{E_2}-\frac{E_3-E_1}{E_0}\right)\right].
\end{align*}
Since $E_0>E_2$ and $E_3+E_1>E_3-E_1>0$, this quantity is strictly positive.

For each triple $(c,f,\alpha)$ and the corresponding constants $B,H,L$ defined above, define the following meromorphic $1$-forms on $\Sigma_g$:
\begin{align*}
\phi_1(c,f,\alpha) &= H \left(\frac{z}{w}\right)^{g-1} \frac{dz}{w} 
        + \frac{1}{c} \eta_1 
        -c\left(\frac{1}{r_{waist}^2} + \frac{\alpha^2}{4|c|^2}\right)B \eta_2 
        - \frac{\alpha^2}{4|c|^2} \frac{r_{waist}^2}{c} \eta_3, \\
\phi_2(c,f,\alpha) &= iH \left(\frac{z}{w}\right)^{g-1} \frac{dz}{w} 
        - i\frac{1}{c} \eta_1 
        - i c \left(\frac{1}{r_{waist}^2} + \frac{\alpha^2}{4|c|^2}\right)B \eta_2 
        + i \frac{\alpha^2}{4|c|^2} \frac{r_{waist}^2}{c} \eta_3, \\
\phi_3(c,f,\alpha) &= L  \eta_{13}, \\
\phi_4(c,f,\alpha) &= \frac{\overline{c}}{f} \left(\frac{1}{r_{waist}^2} + \frac{\alpha^2}{4|c|^2}\right)
           H \left(\frac{z}{w}\right)^{g-1} \frac{dz}{w}
          + \overline{f} B\eta_2
          + \frac{f}{|c|^2} r_{waist}^2 \eta_3, \\
\phi_5(c,f,\alpha) &= -i \frac{\overline{c}}{f} \left(\frac{1}{r_{waist}^2} + \frac{\alpha^2}{4|c|^2}\right) 
           H \left(\frac{z}{w}\right)^{g-1} \frac{dz}{w}
           - i  \overline{f} B \eta_2
           + i \frac{f}{|c|^2} r_{waist}^2  \eta_3.
\end{align*}

The above meromorphic $1$-forms have no common zeros on $\Sigma_g$. Indeed, we have
\begin{align*}
\phi_4(c,f,\alpha)-i\phi_5(c,f,\alpha)=\frac{2f}{|c|^2}r_{waist}^2\eta_3.
\end{align*}
Thus, any common zero must be a zero of $\eta_3$. Comparing the zeros of $\eta_3$ with those of $\phi_3(c,f,\alpha)=L\eta_{13}$, we see that the only possible common zero is $\infty$. However, at $\infty$, 
\begin{align*}
\left(\frac{z}{w}\right)^{g-1}\frac{dz}{w}\quad\text{and}\quad \eta_3
\end{align*}
are holomorphic, whereas $\eta_2$ has a pole of order two. Consequently, $\phi_4(c,f,\alpha)$ has a pole at $\infty$, ruling out the only possible common zero.
 
In Section \ref{main}, we will show that the Weierstrass representation determined by these forms yields a well-defined minimal immersion of $\Sigma_g\setminus\{q_1,q_2,q_3\}$ into $\mathbb{R}^5$, whose image is a surface $\mathsf{S_g}(c,f,\alpha)$ satisfying all the properties stated in the main theorem.

\subsection{Idea of constructing Weierstrass data}
When the underlying Riemann surface $\Sigma_g$ is explicitly given, one can use (\ref{guide}) to search for suitable Weierstrass data. In the characterization of the Costa--Hoffman--Meeks surfaces in $\mathbb{R}^3$ \cite{HM1}, the underlying Riemann surface is determined by assuming sufficiently many symmetries, for instance, at least $4(g+1)$. Indeed, let $H$ denote the group of orientation-preserving symmetries that fix each end. Then $H$ is a finite cyclic group. Under this symmetry assumption, the Riemann--Hurwitz formula shows that the quotient $\Sigma_g/H$ is a sphere. Moreover, in codimension one, embeddedness determines the orientation of each end, which in turn determines the rotation numbers associated with the action of $H$. This allows one to identify a cyclic branched covering equivalent to $\Sigma_g\to\Sigma_g/H$.

A similar argument can be applied in higher codimensions. However, embeddedness no longer uniquely determines the orientations of the ends, leaving several possibilities for the underlying Riemann surface. In \cite{JLEY2}, which addresses the desingularization problem for the Lagrangian catenoid in $\mathbb{R}^4$, all possible cases were studied, leading to nonexistence results under certain assumptions on the number of symmetries. In the present paper, rather than investigating all possible cases, we focus on establishing existence by fixing the underlying Riemann surface to be the same as that of the Costa--Hoffman--Meeks surfaces, namely,
\begin{align*}
\Sigma_g=\left\{(z, w)\in\left(\mathbb{C}\cup\{\infty\}\right)\times\left(\mathbb{C}\cup\{\infty\}\right)\ \middle|\ w^{g+1}=z^g(z+1)(z-1)\right\}.
\end{align*}
We also require that the intrinsic symmetries induced by ambient isometries agree with those of the Costa--Hoffman--Meeks surfaces.

On the Riemann surface $\Sigma_g$ defined above, we choose the meromorphic $1$-forms $\eta_1,\eta_2,\eta_3,\eta_{13}$ as in Section \ref{GUGUGU}, which satisfy conditions (\ref{shshsh}) and (\ref{hshshs}). Moreover, the holomorphic $1$-forms on $\Sigma_g$ are given by
\begin{align*}
h\left(\frac{z}{w}\right)\frac{dz}{w}
\end{align*}
for some polynomial $h(t)$ of degree at most $g-1$. Thus, by (\ref{guide}), we may seek suitable Weierstrass data as linear combinations of these differentials. Since each $\eta_j$ contributes the leading-order term at the corresponding puncture $q_j$, its coefficient vector determines the image of $q_j$ under the generalized Gauss map. As we are interested in the Hoffman--Osserman catenoid $\mathsf{HO}(c,f,\alpha)$ and the plane $\mathsf{P}(c,f,\alpha)$ containing its waist circle, we can use their geometric data to determine these coefficient vectors, possibly up to complex conjugation depending on the orientations of the ends.

To resolve this ambiguity, we apply the symmetry argument discussed above. Let $H$ denote the group of orientation-preserving symmetries that fix each end, and consider the case in which the surface has at least $4(g+1)$ symmetries. Under this assumption, the quotient $\Sigma_g/H$ is a sphere, and the quotient map $\Sigma_g\to\Sigma_g/H$ is equivalent to a cyclic branched covering of $\mathbb{S}^2$. Such a covering is determined by the rotation numbers associated with the action of $H$.

In our setting, however, the underlying Riemann surface $\Sigma_g$ and the corresponding quotient map have already been fixed. We can therefore reverse the argument and use the corresponding rotation numbers to determine the orientations of the ends. This further restricts the coefficient vectors of $\eta_1,\eta_2,\eta_3$ in the Weierstrass data. For details on the computation of rotation numbers in higher codimensions, we refer the reader to \cite[Section $5$]{JLEY2}.

However, considerable freedom remains in the coefficient vector of $\eta_{13}$ and the holomorphic differential part of the Weierstrass data. To narrow down these possibilities, we further restrict the coefficients by requiring the symmetry relations established in Lemma \ref{lambda} and Lemma \ref{kappa}. Different choices of the ambient isometries $\Lambda$ and $K$, however, lead to different possibilities. Guided by the symmetries of the Costa--Hoffman--Meeks surfaces, we define $\Lambda_{c,f,\alpha}$ using the reflection $U_{c,f,\alpha}$ across the plane $\mathsf{P}_0(c,f,\alpha)$ through the origin and parallel to $\mathsf{P}(c,f,\alpha)$, and choose $K_{c,f}$ accordingly. After simplifying the coefficients in this way, we impose the remaining conditions on the Weierstrass data, including conformality and the period conditions, to obtain their final form.

\section{Proof of the Main Theorem}\label{main}
\setcounter{equation}{0}
In this section, we prove that the Weierstrass data constructed in the previous section give rise to well-defined minimal immersions satisfying all the properties of $\mathsf{S_g}(c,f,\alpha)$ stated in the main theorem. The proof is organized into several subsections.

\subsection{The conformality condition}
We verify that the sum of the squares of the meromorphic $1$-forms $\phi_j(c,f,\alpha)$ defined in Section \ref{GUGUGU} vanishes identically. 

\begin{lemma}\label{squared sum}
The meromorphic $1$-forms $\phi_k(c,f,\alpha)$ defined in Section \ref{GUGUGU} satisfy the conformality condition
\begin{align*}
\sum_{k=1}^5\phi_k(c,f,\alpha)^2\equiv 0
\end{align*}
for every $\alpha>0$ and $c,f\in\mathbb{C}\setminus\{0\}$.
\end{lemma}

\begin{proof}
By direct calculation, we have
\begin{align*}
\sum_{k=1}^{5} \phi_k(c,f,\alpha)^2
=
&\left[L^2 \frac{1}{(z+1)^2 (z-1)^2}+\frac{4H}{c} \frac{1}{z(z+1)(z-1)^2}\right.\\
&- \frac{\alpha^2}{|c|^2} \frac{r_{waist}^2}{c} H \frac{1}{z(z+1)^2 (z-1)}
\\
&- 4B \left( \frac{1}{r_{waist}^2} + \frac{\alpha^2}{4|c|^2}\right)
\frac{1}{(z+1)(z-1)^2}
\\
&+ B \frac{\alpha^2}{|c|^2}r_{waist}^2
\left( \frac{1}{r_{waist}^2} + \frac{\alpha^2}{4|c|^2} \right)
\frac{1}{(z+1)^2 (z-1)}\\
&+ 4|f|^2 \frac{r_{waist}^2}{|c|^2} B \frac{1}{(z+1)^2 (z-1)}
\\
&\left.+ 4 \left( \frac{1}{r_{waist}^2} + \frac{\alpha^2}{4|c|^2} \right)\frac{r_{waist}^2}{c}
H \frac{1}{z(z+1)^2 (z-1)}
\right] dz^2 .
\end{align*}
Putting the terms over a common denominator $z(z+1)^2 (z-1)^2$, we obtain
\begin{align*}
\sum_{k=1}^{5} \phi_k(c,f,\alpha)^2
=
&\left[L^2 z+ \frac{4H}{c}(z+1)- \frac{\alpha^2}{|c|^2}\frac{r_{waist}^2}{c} H (z-1)\right.\\
&- 4B \left( \frac{1}{r_{waist}^2} + \frac{\alpha^2}{4|c|^2}\right) z(z+1)
\\
&+ B \frac{\alpha^2}{|c|^2}r_{waist}^2\left( \frac{1}{r_{waist}^2} + \frac{\alpha^2}{4|c|^2} \right) z(z-1)\\
&+ 4|f|^2 \frac{r_{waist}^2}{|c|^2} B z(z-1)
\\
&\left.+ \frac{\alpha^2}{|c|^2} \frac{r_{waist}^2}{c} H (z-1)
+ \frac{4H}{c} (z-1)
\right] \frac{1}{z(z+1)^2 (z-1)^2}dz^2 .
\end{align*}

The coefficient of $z^2$ in the bracket is
\begin{align*}
&-4B \left( \frac{1}{r_{waist}^2} + \frac{\alpha^2}{4|c|^2} \right)
+ B \frac{\alpha^2}{|c|^2}r_{waist}^2
\left( \frac{1}{r_{waist}^2} + \frac{\alpha^2}{4|c|^2} \right)
+ 4|f|^2 \frac{r_{waist}^2}{|c|^2} B
\\
&= \frac{B}{r_{waist}^2|c|^2} \left[
- 4|c|^2
+ \frac{\alpha^4}{4|c|^2}r_{waist}^4
+ 4|f|^2r_{waist}^4
\right]\\
&=0,
\end{align*}
where we used
\begin{align*}
r_{waist}^4
=
\frac{16|c|^4}{16|c|^2 |f|^2 + \alpha^4}.
\end{align*}

Similarly, the coefficient of $z$ in the bracket is
\begin{align*}
&L^2+\frac{8H}{c}
- 4B \left[
\frac{1}{r_{waist}^2}+\frac{\alpha^2}{4|c|^2}
+\frac{\alpha^2}{4|c|^2}r_{waist}^2\left(\frac{1}{r_{waist}^2}+\frac{\alpha^2}{4|c|^2}\right)
+ |f|^2\frac{r_{waist}^2}{|c|^2}
\right]\\
&= L^2+\frac{8H}{c}\\
&\qquad- 4B\left(\frac{1}{r_{waist}^2}+\frac{\alpha^2}{4|c|^2}\right)\left[\left(1+\frac{\alpha^2}{4|c|^2}r_{waist}^2\right)+ r_{waist}^2\left(\frac{1}{r_{waist}^2}-\frac{\alpha^2}{4|c|^2}\right)
\right]\\
&= L^2+\frac{8H}{c}
- 8B \left(
\frac{1}{r_{waist}^2}
+ \frac{\alpha^2}{4|c|^2}
\right)\\
&=0.
\end{align*}
In the first equality, we used
\begin{align*}
\frac{|f|^2}{|c|^2}=\frac{1}{r_{waist}^4}-\frac{\alpha^4}{16|c|^4}=\left(\frac{1}{r_{waist}^2}-\frac{\alpha^2}{4|c|^2}\right)\left(\frac{1}{r_{waist}^2}+\frac{\alpha^2}{4|c|^2}\right),
\end{align*} 
which follows from the definition of $r_{waist}$. The last equality follows from the definition of $L$. 

The constant term in the bracket also vanishes. This completes the proof of the lemma.
\end{proof}

\subsection{Symmetries}
Let $\alpha>0$ and $c,f\in\mathbb{C}\setminus\{0\}$. The plane through the origin and parallel to $\mathsf{P}(c,f,\alpha)$ is spanned by the two vectors
\begin{align*}
\left(\frac{1}{r_{waist}^2}+\frac{\alpha^2}{4|c|^2}\right)\left(-\textup{Re}\ c,\textup{Im}\ c,0,0,0\right)+\left(0,0,0,\textup{Re}\ f,-\textup{Im}\ f\right)
\end{align*}
and
\begin{align*}
\left(\frac{1}{r_{waist}^2}+\frac{\alpha^2}{4|c|^2}\right)\left(\textup{Im}\ c,\textup{Re}\ c,0,0,0\right)+\left(0,0,0,\textup{Im}\ f,\textup{Re}\ f\right).
\end{align*}
We denote this plane by $\mathsf{P}_0(c,f,\alpha)$. Recall that $\mathsf{P}(c,f,\alpha)$ denotes the plane containing the waist circle of the Hoffman--Osserman catenoid $\mathsf{HO}(c,f,\alpha)$; see Section \ref{dubidubab}.

Let $U_{c,f,\alpha}\in O(5)$ denote the reflection across $\mathsf{P}_0(c,f,\alpha)$. The orthogonal complement $\mathsf{P}_0(c,f,\alpha)^{\perp}$ is spanned by the orthogonal vectors
\begin{align*}
&\left(1,0,0,\left(\frac{1}{r_{waist}^2}+\frac{\alpha^2}{4|c|^2}\right)\frac{\textup{Re}(cf)}{|f|^2},-\left(\frac{1}{r_{waist}^2}+\frac{\alpha^2}{4|c|^2}\right)\frac{\textup{Im}(cf)}{|f|^2}\right),\\
&\left(0,1,0,-\left(\frac{1}{r_{waist}^2}+\frac{\alpha^2}{4|c|^2}\right)\frac{\textup{Im}(cf)}{|f|^2},-\left(\frac{1}{r_{waist}^2}+\frac{\alpha^2}{4|c|^2}\right)\frac{\textup{Re}(cf)}{|f|^2}\right),
\end{align*}
and 
\begin{align*}
(0,0,1,0,0).
\end{align*}
Consequently, $U_{c,f,\alpha}$ is given by
\begin{align*}
U_{c,f,\alpha}=\begin{pmatrix}
\displaystyle \frac{\alpha^2}{4|c|^2} r_{waist}^2\, I
& &
\displaystyle \frac{r_{waist}^2}{|c|^2}
\begin{pmatrix}
-\mathrm{Re}(cf) & \mathrm{Im}(cf) \\
\mathrm{Im}(cf) & \mathrm{Re}(cf)
\end{pmatrix}
\\
&-1&\\
\displaystyle \frac{r_{waist}^2}{|c|^2}
\begin{pmatrix}
-\mathrm{Re}(cf) & \mathrm{Im}(cf) \\
\mathrm{Im}(cf) & \mathrm{Re}(cf)
\end{pmatrix}
& &
- \displaystyle \frac{\alpha^2 }{4|c|^2}r_{waist}^2 I
\end{pmatrix}.
\end{align*}
Here, $I$ denotes the $2\times 2$ identity matrix, and all blank entries are zero. From this expression, one can readily verify that $U_{c,f,\alpha}$ exchanges the limiting unoriented tangent planes of the two ends of $\mathsf{HO}(c,f,\alpha)$.

Define
\begin{align*}
\Lambda_{c,f,\alpha} :=
\begin{pmatrix}
\mathcal{S}_{\frac{\pi}{g+1}} & 
\\ & 1 &
\\ & &\mathcal{S}_{-\frac{\pi}{g+1}}
\end{pmatrix}\cdot U_{c,f,\alpha}
\end{align*}
and
\begin{align*}
K_{c,f} := 
\begin{pmatrix}
\mathcal{T}_{2\theta_c}  & 
\\ & 1 &
\\ & &\mathcal{T}_{2\theta_f} 
\end{pmatrix}.
\end{align*}
We remark that the block matrix in the definition of $\Lambda_{c,f,\alpha}$ and the matrix $K_{c,f}$ are of the forms given in (\ref{aa}).

The Riemann surface $\Sigma_g$ is invariant under the following conformal mappings of $\left(\mathbb{C}\cup\{\infty\}\right)^2$:
\begin{align*}
\lambda(z,w):=\left(-z,\mu^gw\right),\quad \kappa(z,w):=(\overline{z},\overline{w}),
\end{align*}
where $\mu:=e^{\frac{i\pi}{g+1}}$. Hence, $\lambda$ and $\kappa$ generate a group of conformal diffeomorphisms of $\Sigma_g$ isomorphic to the dihedral group of order $4(g+1)$. We now compute how $\lambda$ and $\kappa$ act on the Weierstrass data
\begin{align*}
\Phi_{c,f,\alpha}:=\left(\phi_1(c,f,\alpha),\phi_2(c,f,\alpha),\phi_3(c,f,\alpha),\phi_4(c,f,\alpha),\phi_5(c,f,\alpha)\right)
\end{align*}
defined in Section \ref{GUGUGU}.

\begin{lemma}\label{lambda}
For every $\alpha>0$ and $c,f\in\mathbb{C}\setminus\{0\}$, we have
\begin{align*}
\lambda^*\Phi_{c,f,\alpha}=\Lambda_{c,f,\alpha}\Phi_{c,f,\alpha}.
\end{align*}
\end{lemma}

\begin{proof}
For convenience, define
\begin{align*}
Q_0&:=\left(1,i,0,\frac{\overline{c}}{f}\left(\frac{1}{r_{waist}^2}+\frac{\alpha^2}{4|c|^2}\right),-i\frac{\overline{c}}{f}\left(\frac{1}{r_{waist}^2}+\frac{\alpha^2}{4|c|^2}\right)\right),\\
Q_1&:=(1,-i,0,0,0),\\
Q_2&:=\left(-c\left(\frac{1}{r_{waist}^2}+\frac{\alpha^2}{4|c|^2}\right),-ic\left(\frac{1}{r_{waist}^2}+\frac{\alpha^2}{4|c|^2}\right),0,\overline{f},-i\overline{f}\right),\\
Q_3&:=\left(-\frac{\alpha^2}{4|c|^2},i\frac{\alpha^2}{4|c|^2},0,\frac{f}{\overline{c}},i\frac{f}{\overline{c}}\right),\\
Q_4&:=(0,0,1,0,0).
\end{align*}
Then $\Phi_{c,f,\alpha}$ can be expressed as
\begin{align*}
\Phi_{c,f,\alpha}=H\left(\frac{z}{w}\right)^{g-1}\frac{dz}{w}Q_0+\frac{1}{c}\eta_1Q_1+B\eta_2Q_2+\frac{r_{waist}^2}{c}\eta_3Q_3+L\eta_{13}Q_4.
\end{align*}

A direct computation yields
\begin{align*}
\Lambda_{c,f,\alpha}=\begin{pmatrix}
\dfrac{\alpha^2}{4 |c|^2}r_{waist}^2 \mathcal{S}_{\frac{\pi}{g+1}} &
 &
-\dfrac{|f|}{|c|} r_{waist}^2 \, \mathcal{T}_{-\frac{\pi}{g+1}+\theta_c+\theta_f}
\\
 &
-1 &
&
\\
-\dfrac{|f|}{|c|} r_{waist}^2 \, \mathcal{T}_{\frac{\pi}{g+1}+\theta_c+\theta_f} &
 &
-\dfrac{\alpha^2}{4|c|^2} r_{waist}^2 \mathcal{S}_{-\frac{\pi}{g+1}}
\end{pmatrix}.
\end{align*}
It follows that
\begin{align*}
\Lambda_{c,f,\alpha}Q_0=-\frac{1}{\mu}Q_0,\quad \Lambda_{c,f,\alpha}Q_2=\frac{1}{\mu}Q_2,\quad \Lambda_{c,f,\alpha}Q_4=-Q_4,
\end{align*}
and
\begin{align*}
\Lambda_{c,f,\alpha}Q_1=-\mu r_{waist}^2Q_3,\quad \Lambda_{c,f,\alpha}Q_3=-\frac{\mu}{r_{waist}^2}Q_1.
\end{align*}
Consequently,
\begin{align*}
&\Lambda_{c,f,\alpha}\Phi_{c,f,\alpha}\\
&=-\frac{1}{\mu}H\left(\frac{z}{w}\right)^{g-1}\frac{dz}{w}Q_0-\mu\frac{r_{waist}^2}{c}\eta_1Q_3+\frac{1}{\mu}B\eta_2Q_2-\mu\frac{1}{c}\eta_3Q_1-L\eta_{13}Q_4.
\end{align*}

On the other hand, we have
\begin{align*}
\lambda^*\left(\left(\frac{z}{w}\right)^{g-1}\frac{dz}{w}\right)=-\frac{1}{\mu}\left(\frac{z}{w}\right)^{g-1}\frac{dz}{w},\quad \lambda^*\eta_2=\frac{1}{\mu}\eta_2,\quad \lambda^*\eta_{13}=-\eta_{13},
\end{align*}
and
\begin{align*}
\lambda^*\eta_1=-\mu\eta_3,\quad \lambda^*\eta_3=-\mu\eta_1.
\end{align*}
Hence,
\begin{align*}
&\lambda^*\Phi_{c,f,\alpha}\\
&=-\frac{1}{\mu}H\left(\frac{z}{w}\right)^{g-1}\frac{dz}{w}Q_0-\mu\frac{1}{c}\eta_3Q_1+\frac{1}{\mu}B\eta_2Q_2-\mu\frac{r_{waist}^2}{c}\eta_1Q_3-L\eta_{13}Q_4.
\end{align*}
Comparing the two expressions, we conclude that
\begin{align*}
\lambda^*\Phi_{c,f,\alpha}=\Lambda_{c,f,\alpha}\Phi_{c,f,\alpha}, 
\end{align*}
which completes the proof.
\end{proof}

\begin{lemma}\label{kappa}
For every $\alpha>0$ and $c,f\in\mathbb{C}\setminus\{0\}$, we have
\begin{align*}
\overline{\kappa^*\Phi_{c,f,\alpha}}=K_{c,f}\Phi_{c,f,\alpha}.
\end{align*}
\end{lemma}

\begin{proof}
We use the same notation as in the proof of Lemma \ref{lambda}. A direct computation gives
\begin{align*}
K_{c,f}Q_0=\frac{\overline{c}}{c}\overline{Q_0},\quad K_{c,f}Q_1=\frac{c}{\overline{c}}\overline{Q_1},\quad K_{c,f}Q_2=\overline{Q_2},\quad K_{c,f}Q_3=\frac{c}{\overline{c}}\overline{Q_3},
\end{align*}
and 
\begin{align*}
K_{c,f}Q_4=Q_4=\overline{Q_4}.
\end{align*}
Recalling that
\begin{align*}
\Phi_{c,f,\alpha}=H\left(\frac{z}{w}\right)^{g-1}\frac{dz}{w}Q_0+\frac{1}{c}\eta_1Q_1+B\eta_2Q_2+\frac{r_{waist}^2}{c}\eta_3Q_3+L\eta_{13}Q_4,
\end{align*}
we obtain
\begin{align*}
K_{c,f}\Phi_{c,f,\alpha}=\frac{\overline{c}}{c}H\left(\frac{z}{w}\right)^{g-1}\frac{dz}{w}\overline{Q_0}+\frac{1}{\overline{c}}\eta_1\overline{Q_1}+B\eta_2\overline{Q_2}+\frac{r_{waist}^2}{\overline{c}}\eta_3\overline{Q_3}+L\eta_{13}\overline{Q_4}.
\end{align*}

On the other hand, we have
\begin{align*}
\overline{\kappa^*\Phi_{c,f,\alpha}}=\overline{H}\left(\frac{z}{w}\right)^{g-1}\frac{dz}{w}\overline{Q_0}+\frac{1}{\overline{c}}\eta_1\overline{Q_1}+\overline{B}\eta_2\overline{Q_2}+\frac{r_{waist}^2}{\overline{c}}\eta_3\overline{Q_3}+\overline{L}\eta_{13}\overline{Q_4}.
\end{align*}
By the definitions of $H$, $B$, and $L$, 
\begin{align*}
\overline{H}=\frac{\overline{c}}{c}H,\quad \overline{B}=B,\quad \overline{L}=L.
\end{align*}
Therefore,
\begin{align*}
\overline{\kappa^*\Phi_{c,f,\alpha}}=\frac{\overline{c}}{c}H\left(\frac{z}{w}\right)^{g-1}\frac{dz}{w}\overline{Q_0}+\frac{1}{\overline{c}}\eta_1\overline{Q_1}+B\eta_2\overline{Q_2}+\frac{r_{waist}^2}{\overline{c}}\eta_3\overline{Q_3}+L\eta_{13}\overline{Q_4}.
\end{align*}
Comparing the two expressions, we conclude that
\begin{align*}
\overline{\kappa^*\Phi_{c,f,\alpha}}=K_{c,f}\Phi_{c,f,\alpha}.
\end{align*}
This completes the proof.
\end{proof}

\subsection{Period computations}\label{period}
In this subsection, we show that the meromorphic $1$-forms defined in Section \ref{GUGUGU} have vanishing real periods on $\Sigma_g\setminus\{q_1,q_2,q_3\}$. 

We begin with $\phi_3(c,f,\alpha)$. Since
\begin{align*}
\phi_3(c,f,\alpha)=\frac{L}{(z+1)(z-1)}dz=\frac{L}{2}d\log\left(\frac{z-1}{z+1}\right),
\end{align*}
and $L$ is real, we have
\begin{align*}
\textup{Re}\int_{\delta}\phi_3(c,f,\alpha)=0
\end{align*}
for every closed curve $\delta$ in $\Sigma_g\setminus\{q_1,q_2,q_3\}$. 

Moreover, the holomorphic differential
\begin{align*}
\left(\frac{z}{w}\right)^{g-1}\frac{dz}{w}
\end{align*}
has vanishing residues. As discussed in Section \ref{emenmulone}, the residues of 
\begin{align*}
\eta_1,\quad \eta_2,\quad \eta_3
\end{align*}
also vanish at the three punctures $q_1$, $q_2$, and $q_3$. Consequently, the linear combinations
\begin{align*}
\phi_1(c,f,\alpha),\quad \phi_2(c,f,\alpha),\quad \phi_4(c,f,\alpha),\quad \phi_5(c,f,\alpha)
\end{align*}
have vanishing residues at all three punctures. Thus, it remains to show that $\phi_1(c,f,\alpha)$, $\phi_2(c,f,\alpha)$, $\phi_4(c,f,\alpha)$, and $\phi_5(c,f,\alpha)$ have vanishing real periods along every closed curve in $\Sigma_g$. 

Since our Riemann surface $\Sigma_g$ is the same as that considered in \cite{HM1, JLEY2}, we may use the homology basis and the branch of $\left( z^g (z+1)(z-1) \right)^{\frac{1}{g+1}}$ introduced therein to evaluate the period integrals. Define 
\begin{align*}
\gamma(t)=\frac{1}{2}+e^{it},\quad 0\leq t\leq 2\pi,
\end{align*}
and let $\tilde{\gamma}(t)=(\gamma(t),w(t))$ denote the lift to $\Sigma_g$ with $w(0)\in\mathbb{R}$. Then the images of $\tilde{\gamma}$ under the group generated by the conformal diffeomorphisms $\lambda$ and $\kappa$ generate $H_1(\Sigma_g;\mathbb{Z})$. By Lemma \ref{lambda} and Lemma \ref{kappa}, it suffices to verify the vanishing of the real periods along $\tilde{\gamma}$. 

We first recall the following integral identities.
\begin{lemma}[Lemma $6.13$ in \cite{JLEY2}]\label{BOONGBOONG}
We have
\begin{align*}
&\int_{\tilde{\gamma}}\left(\frac{z}{w}\right)^{g-1}\frac{dz}{w}=2i\sin\left(\frac{\pi}{g+1}\right)E_0,\quad \int_{\tilde{\gamma}}\eta_1=-2i\sin\left(\frac{\pi}{g+1}\right)E_1,\\
&\int_{\tilde{\gamma}}\eta_2=2i\sin\left(\frac{\pi}{g+1}\right)E_2,\quad \int_{\tilde{\gamma}}\eta_3=2i\sin\left(\frac{\pi}{g+1}\right)E_3,
\end{align*}
where $E_0,E_1,E_2,E_3$ are the constants defined in Section \ref{GUGUGU}.
\end{lemma}
\begin{proof}
See the proof of \cite[Lemma 6.13]{JLEY2}.
\end{proof}

\begin{lemma}\label{period condition}
We have
\begin{align*}
\textup{Re}\int_{\tilde{\gamma}}\phi_1(c,f,\alpha)=\textup{Re}\int_{\tilde{\gamma}}\phi_2(c,f,\alpha)=\textup{Re}\int_{\tilde{\gamma}}\phi_4(c,f,\alpha)=\textup{Re}\int_{\tilde{\gamma}}\phi_5(c,f,\alpha)=0
\end{align*}
for every $\alpha>0$ and $c,f\in\mathbb{C}\setminus\{0\}$.
\end{lemma}
\begin{proof}
Recall that
\begin{align*}
\phi_1(c,f,\alpha)=H \left(\frac{z}{w}\right)^{g-1} \frac{dz}{w}+\frac{1}{c}\eta_1-c\left(\frac{1}{r_{waist}^2} + \frac{\alpha^2}{4|c|^2}\right)B \eta_2-\frac{\alpha^2}{4|c|^2} \frac{r_{waist}^2}{c} \eta_3.
\end{align*}
Using the integral identities in Lemma \ref{BOONGBOONG}, we obtain
\begin{align*}
&\int_{\tilde{\gamma}}\phi_1(c,f,\alpha)\\
&=2i\sin\left(\frac{\pi}{g+1}\right)\left(HE_0-\frac{1}{c}E_1-c\left(\frac{1}{r_{waist}^2}+\frac{\alpha^2}{4|c|^2}\right)BE_2-\frac{\alpha^2}{4|c|^2}\frac{r_{waist}^2}{c}E_3\right).
\end{align*}

Since
\begin{align*}
B=\frac{E_1+E_3}{2|c|^2E_2}r_{waist}^2,\quad H=\frac{E_3-E_1}{2\overline{c}E_0}\left(1-\frac{\alpha^2}{4|c|^2} r_{waist}^2\right),
\end{align*}
it follows that
\begin{align*}
&HE_0-\frac{1}{c}E_1-c\left(\frac{1}{r_{waist}^2}+\frac{\alpha^2}{4|c|^2}\right)BE_2-\frac{\alpha^2}{4|c|^2}\frac{r_{waist}^2}{c}E_3\\
&=\frac{E_3-E_1}{2\overline{c}}\left(1-\frac{\alpha^2}{4|c|^2} r_{waist}^2\right)-\frac{1}{c}E_1-\frac{E_1+E_3}{2\overline{c}}\left(1+\frac{\alpha^2}{4|c|^2}r_{waist}^2\right)-\frac{\alpha^2}{4|c|^2}\frac{r_{waist}^2}{c}E_3\\
&=-\frac{2\textup{Re}\ c}{|c|^2}\left(E_1+\frac{\alpha^2}{4|c|^2}r_{waist}^2E_3\right).
\end{align*}
Hence,
\begin{align*}
\int_{\tilde{\gamma}}\phi_1(c,f,\alpha)=-4i\sin\left(\frac{\pi}{g+1}\right)\frac{\textup{Re}\ c}{|c|^2}\left(E_1+\frac{\alpha^2}{4|c|^2}r_{waist}^2E_3\right).
\end{align*}
In particular,
\begin{align*}
\textup{Re}\int_{\tilde{\gamma}}\phi_1(c,f,\alpha)=0.
\end{align*}

Similarly, we have
\begin{align*}
&\int_{\tilde{\gamma}}\phi_2(c,f,\alpha)=4i\sin\left(\frac{\pi}{g+1}\right)\frac{\textup{Im}\ c}{|c|^2}\left(E_1+\frac{\alpha^2}{4|c|^2}r_{waist}^2E_3\right),\\
&\int_{\tilde{\gamma}}\phi_4(c,f,\alpha)=4i\sin\left(\frac{\pi}{g+1}\right)\frac{\textup{Re}\ f}{|c|^2}r_{waist}^2E_3,\\
&\int_{\tilde{\gamma}}\phi_5(c,f,\alpha)=-4i\sin\left(\frac{\pi}{g+1}\right)\frac{\textup{Im}\ f}{|c|^2}r_{waist}^2E_3.
\end{align*}
These integrals are also purely imaginary, which completes the proof.
\end{proof}

\subsection{Embeddedness}
Combining the results established so far, we conclude that the meromorphic $1$-forms defined in Section \ref{GUGUGU} give rise to a well-defined minimal immersion $X_{c,f,\alpha}:\Sigma_g\setminus\{q_1,q_2,q_3\}\to\mathbb{R}^5$ via the generalized Weierstrass representation (\ref{gwr}). In this subsection, we prove that this immersion is an embedding.

The main idea is to show that the composition of the minimal immersion $X_{c,f,\alpha}$ with a suitable projection coincides with a deformed Costa--Hoffman--Meeks embedding of the same genus. This immediately implies that $X_{c,f,\alpha}$ itself is an embedding. Moreover, if we denote its image by $\mathsf{S_g}(c,f,\alpha)$, then the restriction of the projection to $\mathsf{S_g}(c,f,\alpha)$ is a bijection onto the deformed Costa--Hoffman--Meeks surface. Thus, $\mathsf{S_g}(c,f,\alpha)$ can be viewed as a graph over this surface.

To this end, let
\begin{align*}
\rho:=\frac{\overline{c}}{f}\left(\frac{1}{r_{waist}^2}-\frac{\alpha^2}{4|c|^2}\right),
\end{align*}
and let $W\subset\mathbb{R}^5$ be the three-dimensional subspace spanned by the following orthonormal vectors:
\begin{align*}
\frac{1}{\left(1+|\rho|^2\right)^{\frac{1}{2}}}(1,0,0,-\textup{Re}\ \rho,-\textup{Im}\ \rho),\quad \frac{1}{\left(1+|\rho|^2\right)^{\frac{1}{2}}}(0,1,0,-\textup{Im}\ \rho,\textup{Re}\ \rho),\quad (0,0,1,0,0).
\end{align*}
We identify $W$ with $\mathbb{R}^3$ via this orthonormal basis. Under this identification, the orthogonal projection $\textup{pr}_W:\mathbb{R}^5\to W\simeq \mathbb{R}^3$ is given by
\begin{align*}
&(x_1,x_2,x_3,x_4,x_5)\mapsto\\
&\left(\frac{x_1-\left(\textup{Re}\ \rho\right)x_4-\left(\textup{Im}\ \rho\right)x_5}{\left(1+|\rho|^2\right)^{\frac{1}{2}}},\frac{x_2-\left(\textup{Im}\ \rho\right)x_4+\left(\textup{Re}\ \rho\right)x_5}{\left(1+|\rho|^2\right)^{\frac{1}{2}}},x_3\right).
\end{align*}
We now compute $\textup{pr}_W\Phi_{c,f,\alpha}$, where
\begin{align*}
\Phi_{c,f,\alpha}=\left(\phi_1(c,f,\alpha),\phi_2(c,f,\alpha),\phi_3(c,f,\alpha),\phi_4(c,f,\alpha),\phi_5(c,f,\alpha)\right).
\end{align*}

A direct computation gives
\begin{align*}
\phi_1(c,f,\alpha)&-\left(\textup{Re}\ \rho\right)\phi_4(c,f,\alpha)-\left(\textup{Im}\ \rho\right)\phi_5(c,f,\alpha)\\
&=-\frac{E_1+E_3}{E_2}\left(\frac{1}{\overline{c}}\eta_2-\frac{1}{c}\frac{E_2}{E_1+E_3}(\eta_1-\eta_3)\right)
\end{align*}
and
\begin{align*}
\phi_2(c,f,\alpha)&-\left(\textup{Im}\ \rho\right)\phi_4(c,f,\alpha)+\left(\textup{Re}\ \rho\right)\phi_5(c,f,\alpha)\\
&=-\frac{E_1+E_3}{E_2}\left(\frac{i}{\overline{c}}\eta_2+\frac{i}{c}\frac{E_2}{E_1+E_3}(\eta_1-\eta_3)\right).
\end{align*}
Hence,
\begin{align*}
&-\frac{\left(1+|\rho|^2\right)^{\frac{1}{2}}E_2}{E_1+E_3} \textup{pr}_W\Phi_{c,f,\alpha}\\
&=\left(\frac{1}{\overline{c}}\eta_2-\frac{1}{c}\frac{E_2}{E_1+E_3}(\eta_1-\eta_3),\frac{i}{\overline{c}}\eta_2+\frac{i}{c}\frac{E_2}{E_1+E_3}(\eta_1-\eta_3),-\frac{\left(1+|\rho|^2\right)^{\frac{1}{2}}E_2}{E_1+E_3}L\eta_{13}\right).
\end{align*}

On the other hand, recall that the Weierstrass data for the Costa--Hoffman--Meeks surfaces constructed in \cite{HM1} are given by
\begin{align*}
\Phi_{CHM}:=\left(\eta_2-\frac{E_2}{E_1+E_3}(\eta_1-\eta_3),i\eta_2+i\frac{E_2}{E_1+E_3}(\eta_1-\eta_3),\left(\frac{8E_2}{E_1+E_3}\right)^{\frac{1}{2}}\eta_{13}\right)
\end{align*}
on $\Sigma_g$. Applying the rotation
\begin{align*}
\begin{pmatrix}
\mathcal{S}_{-\theta_c} & \\
 & 1\end{pmatrix}
\end{align*}
followed by a scaling by the factor $\frac{1}{|c|}$, we obtain the Weierstrass data
\begin{align*}
\left(\frac{1}{\overline{c}}\eta_2-\frac{1}{c}\frac{E_2}{E_1+E_3}(\eta_1-\eta_3),\frac{i}{\overline{c}}\eta_2+\frac{i}{c}\frac{E_2}{E_1+E_3}(\eta_1-\eta_3),\frac{1}{|c|}\left(\frac{8E_2}{E_1+E_3}\right)^{\frac{1}{2}}\eta_{13}\right).
\end{align*}

Next, consider the linear map $\mathcal{D}:\mathbb{R}^3\to\mathbb{R}^3$ defined by
\begin{align*}
(y_1,y_2,y_3)\mapsto \left(-\frac{E_1+E_3}{\left(1+|\rho|^2\right)^{\frac{1}{2}}E_2}y_1,-\frac{E_1+E_3}{\left(1+|\rho|^2\right)^{\frac{1}{2}}E_2}y_2,|c|\left(\frac{E_1+E_3}{8E_2}\right)^{\frac{1}{2}}Ly_3\right).
\end{align*}
Define
\begin{align*}
\mathcal{A}_{c,f,\alpha}:=\frac{1}{|c|}\mathcal{D}\circ
\begin{pmatrix}
\mathcal{S}_{-\theta_c} & \\
 & 1\end{pmatrix}.
\end{align*}
Combining the above observations gives
\begin{align*}
\textup{pr}_W\Phi_{c,f,\alpha}=\mathcal{A}_{c,f,\alpha}\Phi_{CHM}.
\end{align*} 
If we denote the Costa--Hoffman--Meeks immersion by $X_{CHM}:\Sigma_g\setminus\{q_1,q_2,q_3\}\to\mathbb{R}^3$, then, up to translation, this implies
\begin{align*}
\textup{pr}_W\circ X_{c,f,\alpha}=\mathcal{A}_{c,f,\alpha}\circ X_{CHM}.
\end{align*}
Since $X_{CHM}$ is an embedding, the minimal immersion $X_{c,f,\alpha}$ is also an embedding as explained above. This completes the proof.

\begin{rmk}
\normalfont
The deformed meromorphic $1$-forms $\mathcal{A}_{c,f,\alpha}\Phi_{CHM}$ do not satisfy the conformality condition. Indeed, the squares of the two scaling factors in the definition of $\mathcal{D}$ differ by
\begin{align*}
\frac{(E_1+E_3)^2}{(1+|\rho|^2)E_2^2}-\frac{E_1+E_3}{8E_2}|c|^2L^2=\frac{E_1+E_3}{E_2}\overline{c}H,
\end{align*}
which is nonzero. Consequently, the sum of the squares of the components of $\mathcal{A}_{c,f,\alpha}\Phi_{CHM}$ does not vanish identically. Thus, the parametrizations of the deformed Costa--Hoffman--Meeks surfaces remain harmonic but are no longer conformal. From this perspective, our construction of $\mathsf{S_g}(c,f,\alpha)$ can be viewed as restoring conformality by introducing two additional harmonic coordinate functions.
\end{rmk}

\subsection{Total curvature and asymptotics of the ends}
To complete the proof of the main theorem, we establish the remaining geometric properties of $\mathsf{S_g}(c,f,\alpha)$. We begin by computing its total curvature. Although the argument is standard, we include it for completeness. Let $K_{\mathsf{S_g}(c,f,\alpha)}$ denote
its Gaussian curvature, and consider the divisor
\[
D=2q_1+2q_2+2q_3
\]
on $\Sigma_g$. The meromorphic $1$-forms $\phi_1(c,f,\alpha),\ldots,\phi_5(c,f,\alpha)$ are holomorphic
sections of $C\otimes[D]$, where $C$ denotes the canonical line bundle of $\Sigma_g$.
As discussed at the end of Section \ref{GUGUGU}, they have no common
zeros on $\Sigma_g$.

Consequently, the generalized Gauss map extends holomorphically to $\Sigma_g$.
Its degree $d$ satisfies
\[
d=\deg(C\otimes[D])
 =\deg C+\deg D
 =(2g-2)+6
 =2g+4.
\]
It follows that
\[
\int_{\Sigma_g\setminus\{q_1,q_2,q_3\}} K_{\mathsf{S_g}(c,f,\alpha)}\,dA
=-2\pi d
=-4\pi(g+2).
\]
In particular, the surface has finite total curvature.

We now describe the asymptotic behavior of $\mathsf{S_g}(c,f,\alpha)$ at its three ends.
At each puncture $q_j$, choose a local coordinate $\zeta$ with
$\zeta(q_j)=0$. In this coordinate, the Weierstrass data have the Laurent expansion
\[
(\phi_1(c,f,\alpha),\ldots,\phi_5(c,f,\alpha))
=
\left(
\frac{a_j}{\zeta^2}+\frac{b_j}{\zeta}+h_j(\zeta)
\right)d\zeta,
\]
where $a_j\in\mathbb{C}^5\setminus\{0\}$ and $h_j$ is holomorphic.
The differentials $\eta_1,\eta_2,\eta_3$ have vanishing residues,
whereas
\[
\operatorname{res}_{q_1}\eta_{13}=\frac{g+1}{2},
\quad
\operatorname{res}_{q_3}\eta_{13}=-\frac{g+1}{2}.
\]
Thus, we obtain
\begin{align*}
b_1=\left(0,0,\frac{(g+1)L}{2},0,0\right),\quad b_2=(0,0,0,0,0),\quad b_3=\left(0,0,-\frac{(g+1)L}{2},0,0\right).
\end{align*}

The conformality of the immersion implies that
\[
a_j\cdot a_j=0,
\quad
a_j\cdot b_j=0,
\]
where the dot denotes the complex bilinear extension of the
Euclidean inner product. Integrating the Laurent expansion gives 
\[
X_{c,f,\alpha}(\zeta)
=
\operatorname{Re}\left(-\frac{a_j}{\zeta}\right)
+b_j\log|\zeta|+p_j+O(|\zeta|)
\]
for some $p_j\in\mathbb{R}^5$. Since $a_j\cdot a_j=0$ and $a_j\neq0$, the leading term
parametrizes a plane parallel to
\[
P_j=
\operatorname{span}_{\mathbb{R}}
\{\operatorname{Re}a_j,\operatorname{Im}a_j\}.
\]
Moreover, $a_j\cdot b_j=0$ and $b_j\in\mathbb{R}^5$
imply that $b_j$ is perpendicular to $P_j$.
Thus, the ends at $q_1$ and $q_3$ have nonzero logarithmic growth, with opposite logarithmic coefficients
\[
b_1=-b_3=\left(0,0,\frac{(g+1)L}{2},0,0\right),
\]
whereas the end at $q_2$ is asymptotic to the affine plane
$p_2+P_2$, since $b_2=0$. 

From the coefficients of $\eta_1,\eta_2,\eta_3$ in the
Weierstrass data, we see that $P_1$ and $P_3$ coincide with
the limiting tangent planes of the two ends of $\mathsf{HO}(c,f,\alpha)$, respectively, while $P_2$ is parallel to $\mathsf{P}(c,f,\alpha)$. 

Finally, we show that $\mathsf{S_g}(c,f,\alpha)$ is not contained in any proper affine subspace of $\mathbb{R}^5$. Suppose that its image is contained in an affine hyperplane. Then there exists a nonzero vector $v=(v_1,\ldots,v_5)\in\mathbb{R}^5$ such that
\[
\sum_{j=1}^5 v_j\phi_j(c,f,\alpha)=0.
\]
Comparing the leading coefficients of the Laurent expansions at $q_1$ gives 
\begin{align*}
\frac{1}{c}(v_1-i v_2)=0,
\end{align*}
and hence $v_1=v_2=0$. Comparing the corresponding coefficients at $q_3$ then gives
\[
\frac{f}{|c|^2}r_{waist}^2(v_4+i v_5)=0,
\]
which implies that $v_4=v_5=0$. The remaining identity is $v_3L\eta_{13}=0$. Since $L\neq0$, we obtain $v_3=0$, contradicting the assumption that $v\neq 0$. This completes the proof.

\begin{rmk}\label{rmk57}
\normalfont
The surface $\mathsf{S_g}(c,f,\alpha)$ contains no affine line. This contrasts with the Costa--Hoffman--Meeks surfaces and many other classical minimal surfaces in $\mathbb{R}^3$, which contain straight lines. Indeed, by the preceding analysis of the end asymptotics, any affine line contained in $\mathsf{S_g}(c,f,\alpha)$ must approach the planar end at $q_2$ in both directions. The distance to $p_2+P_2$ therefore tends to zero along both rays, which implies that the entire line lies in $p_2+P_2$. We show that this is impossible. Set
\[
\Omega=\left(\frac{z}{w}\right)^{g-1}\frac{dz}{w}.
\]
In the local coordinate $\zeta=z/w$ near $q_2$, noting that the planar end corresponds to $z \to \infty$, we have
\[
\Omega=\left(-(g+1)\zeta^{g-1}+O\left(\zeta^{3g+1}\right)\right)d\zeta,
\quad \eta_{13}=\zeta\Omega =O\left(\zeta^{g}\right) d\zeta,
\]
and $$\eta_1,\eta_3=O\left(\zeta^{g+1}\right)d\zeta.$$ Let $\pi_2$ denote the orthogonal projection onto $P_2^\perp$. Since the real and imaginary parts of the coefficient vector of $\eta_2$ in the Weierstrass data span $P_2$, the term involving $\eta_2$ vanishes under $\pi_2$. Integrating the Weierstrass data and applying $\pi_2$, we obtain
\begin{align*}
&\left|\pi_2\left(X_{c,f,\alpha}(\zeta)-p_2\right)\right|^2\\
&\qquad =\frac{(g+1)^2}{g^2}|H|^2\left(1+\frac{|c|^2}{|f|^2}\left(\frac{1}{r_{waist}^2}+\frac{\alpha^2}{4|c|^2}\right)^2\right)|\zeta|^{2g}+O\left(|\zeta|^{2g+2}\right).
\end{align*}
Since $H\neq0$, the leading coefficient is strictly positive. Thus, the planar end is disjoint from $p_2+P_2$ for all sufficiently small nonzero $\zeta$, contradicting the existence of such a line.
\end{rmk}


\end{document}